\documentclass[11pt,a4paper]{amsart}
\usepackage{geometry, empheq}
\usepackage[english]{babel}
\usepackage{amsmath}
\usepackage{amssymb}
\usepackage{amsthm}
\usepackage{xcolor,graphicx}
\usepackage{color}
\usepackage{varioref}
\usepackage{nicefrac}
\usepackage{caption}
\usepackage{bbm}
\usepackage{subcaption}
\usepackage{float}
\usepackage[normalem]{ulem}
\usepackage{multirow}
\usepackage{enumerate}
\usepackage{algorithm}
\usepackage{algorithmic}
\usepackage[textsize=scriptsize]{todonotes}

\newcommand{\norm}[1]{{\|  #1\|}}

\newcommand{\Lvec}{\mathbf{L}^2(\Omega)}
\newcommand{\Lmat}{\mathbb{L}^2(\Omega)}
\newcommand{\Hvec}{\mathbf{H}^1(\Omega)}
\newcommand{\om}{\Omega}

\newcommand{\reals}{\mathbb{R}}
\newcommand{\abs} [1] {\left\vert#1\right\vert}

\newcounter{theassumption}
\newtheorem{assumption}[theassumption]{{Assumption}}
\newtheorem{proposition}{{Proposition}}[section]
\newtheorem{definition}{{Definition}}[section]
\newtheorem{lemma}{{Lemma}}[section]
\newtheorem{remark}{{Remark}}[section]

\newtheorem{theorem}{{Theorem}}[section]
\numberwithin{equation}{section}

\allowdisplaybreaks

\begin{document}
\title[Optimal control of a shear--thickening fluid]{Optimal control of a nonsmooth PDE arising in the modeling of shear--thickening fluids}
\author[J.C. De los Reyes and P. Quiloango]{Juan Carlos De los Reyes$^\ddag$ and Paola Quiloango$^\ddag$}
\address{$^\ddag$Research Center for Mathematical Modeling (MODEMAT) and Department of Mathematics, Escuela Polit\'ecnica Nacional, Quito, Ecuador}
\thanks{$^*$This research has been supported by Escuela Polit\'ecnica Nacional de Ecuador under award PIMI-16-14: \emph{Modelización matemática y control de fluidos magneto y electro-reológicos}. Paola Quiloango acknowledges partial support from the \emph{Master Program in Mathematical Optimization}, Escuela Politécnica Nacional de Ecuador. Juan Carlos De los Reyes acknowledges partial support of the Weierstrass Institute for Applied Analysis and Stochastics, Berlin-Germany.}

\smallskip
\begin{abstract}
This paper focuses on the analysis of an optimal control problem governed by a nonsmooth quasilinear partial differential equation that models a stationary incompressible shear-thickening fluid. We start by studying the directional differentiability of the non-smooth term within the state equation as a prior step
to demonstrate the directional differentiability of the solution operator. Thereafter, we establish a primal first order necessary optimality condition (\emph{Bouligand (B) stationarity}), which is derived from the directional differentiability of the solution operator. By using a local regularization of the nonsmooth term and carrying out an asymptotic analysis thereafter, we rigourously derive a \emph{weak stationarity system} for local minima. By combining the B- and weak stationarity conditions, and using the regularity of the Lagrange multiplier, we are able to obtain a \emph{strong stationarity system} that includes an inequality for the scalar product between the symmetrized gradient of the state and the Lagrange multiplier.
\end{abstract}

\maketitle


\section{Introduction}
Optimal control problems of partial differential equations (PDE) involving nonsmooth nonlinear terms pose several analytical challenges as, differently from the smooth cases, there is no adjoint calculus available to derive optimality conditions in a straightforward manner \cite{troeltzsch,delosreyesbook}. To overcome these difficulties, in recent years alternative methodologies have been proposed for the study of these problems, mainly using the techniques developed for the analysis of control problems constrained by variational inequalities. In particular, the directional differentiability of the solution operator constitutes a corner stone of current approaches.

The directional differentiability of the solution operator leads, in combination with smooth cost functionals, to primal optimality conditions for the control problem, the so-called \emph{B-stationarity conditions} (see, e.g., \cite{luopang1996}). However, to prove such directional differentiability in the infinite dimensional setting, a careful study of the nonsmoothness and the regularity of the control and state variables has to be carried out.

In order to obtain optimality conditions involving also dual variables, i.e., Lagrange multipliers, additional results are required. In \cite{meyersusu2017} a regularization approach was proposed for getting existence of Lagrange multipliers for the nonsmooth optimal control problem. This smoothing technique leads only to \emph{weak stationarity conditions} in the limit. However, by combining both the B-stationarity conditions and the weak stationarity ones, \emph{strong optimality conditions} for the nonsmooth control problem may be obtained. This approach has been already studied for the control of nonsmooth semilinear \cite{meyersusu2017,christof2018} and some quasilinear equations \cite{clason2021}.

In this article, we consider a nonsmooth optimal control problem governed by an elliptic quasilinear equation that models a stationary incompressible shear-thickening fluid with homogeneous Dirichlet boundary conditions. Differently from \cite{clason2021}, the quasilinearity in our case involves the symmetrized gradient of the state variable in suitable solenoidal spaces. This fact makes the analysis more involved, as there is not as much regularity as when the nonsmooth term depends solely on the state \cite{clason2021} and, therefore, many compact embeddings are missing. By using monotonicity arguments, we are able compensate the lack of compactness and obtain approximation results for the regularized problems (Theorem \ref{th:convergenceofminimizers}), as well as a limiting weak stationarity system (Theorem \ref{th:optimalityafterlimit}). Moreover, after proving the directional differentiability of the solution operator, we are able to prove a B-stationarity result. Combining the B- and weak stationarity systems, we are able to obtain a strong stationary system (Theorem \ref{th:strongstationarity}). The main novelty of this system is the inclusion of a pointwise inequality relation for the scalar product between the Lagrange multiplier and the symmetrized gradient of the state on the active set.

Specifically, we consider the optimal control problem
\begin{subequations}\label{eq:controlproblem}
	\begin{align} 
	  & \underset{(y,u)\in Y \times U}{\text{min}}   \quad J(y,u) = \dfrac{1}{2} \,   \norm{y-z_{d}}  _{L^{2} (\Omega)^{N}}^{2}+\dfrac{\alpha}{2} \,\norm{u} _{L^{2} (\Omega)^{N}} ^{2} && \label{eq:controlproblem_1}\\
	  \noindent & \text{subject to} && \notag \\&\mu \int_{\Omega} \varepsilon y : \varepsilon v \, dx + \nu \int_{\Omega} \mathsf{m}(\varepsilon y): \varepsilon v \, dx= \int_{\Omega} u \cdot v \, dx,    && \text{for all } v \in Y, \label{eq:controlproblem_2}
	  \end{align}
\end{subequations}
where the nonsmooth function \(\mathsf{m}:\mathbb{R}^{N\times N} \rightarrow \mathbb{R}^{N\times N}\) is defined by
\begin{equation*}
    \mathsf{m}(E)=\begin{cases}
    \text{max}(0, \abs{E}-g)\dfrac{E}{\abs{E}},\quad &\text{if } E \neq 0,\\
    0, \quad &\text{if } E =0.
    \end{cases}
\end{equation*}
\(Y=\left\{ y \in H^1 (\Omega)^N : \nabla \cdot y=0 \text{ in }\Omega , y=0 \text{ on } \Gamma \right\} \) and \(U=L^2(\Omega)^N \). We consider a Lipschitz bounded domain $\Omega \subset \reals ^N$  ($N=2$ or $N=3$), with boundary $\Gamma$, and a desired state $z_d \in L^2(\om)^N$. The symmetrized gradient is denoted by $\varepsilon y$, i.e., $\varepsilon y := \frac12 (\nabla y + \nabla y^T)$,
and we introduce the symmetric tensor \(\mathsf{m}\) in order to explicitely define the nonlinearity in the state equation. Morover, $\mu >0,~\nu \geq 0$ stand for viscosity coefficients of the fluid and $g>0$ for the yield stress at which the viscosity increases. The objective functional \eqref{eq:controlproblem_1} is a tracking type functional, where \(\alpha>0\) is a parameter related to the cost of the control.

The nonlinear model \eqref{eq:controlproblem_2} is related to incompressible shear thickening fluids, which are non-Newtonian supensions characterized by an abrupt increase in viscosity for increasing shear rate. In strong form, the model can be written as
\begin{subequations}
	\begin{align}
		- \operatorname{Div} \sigma &=u && \text{ in }\Omega,\\
		\nabla \cdot y & =0 && \text{ in } \Omega,\\
		 y&=0 && \text{ on } \Gamma,
	\end{align}
\end{subequations}
where $\sigma$ stand for the total stress tensor $\sigma := 2 \mu \varepsilon y + 2 \nu \mathsf{m}(\varepsilon y) - p \mathbb I$, with $p$ denoting the pressure and $\mathbb I$ the identity matrix.

Optimal control problems for non-Newtonian materials have been previously considered in, e.g., \cite{slawig,wachsmuth,arada,guerra}, where the considered models contain strongly nonlinear differentiable terms. The model proposed in \cite{de2014nonsmooth}, and considered in this article, introduces a non-linearity for which the analysis in classical Sobolev spaces is possible, at the price of having a non-differentiable term in the equation.

\emph{The organization of the paper is as follows:} In Section \ref{s:differentiability}, we demonstrate the directional differentiability of the solution operator and characterize the derivative by means of a partial differential equation. The optimal control problem is analyzed in Section 3, where existence of a global solution is proved and a primal optimality condition, based on the directional differentiability of the solution operator, is derived. Section 4 is devoted to establishing optimality conditions using a regularization approach. The consistency of the approach is justified and a limiting optimality system is obtained, which corresponds to \emph{weak stationarity}. Finally, in Section 5 we obtain a strong stationarity condition for our problem.

\emph{Notation.} Along the paper we use the notation $\Hvec := H^1(\Omega)^N$, $\Lvec := L^2(\Omega)^N$ and $\Lmat := L^2(\Omega; \mathbb R^{N \times N})$. The notation $|\cdot|$ is used for the absolute value, the Euclidean norm or the Frobenius norm, and the proper one should be inferred from the context. The norm in a normed vector space $X$ is denoted by $\| \cdot \|_X$; the subindex is dismissed in the case of $L^2(\Omega), \Lvec$ or $\Lmat.$


\section{Directional differentiablity of the solution operator}\label{s:differentiability}
We begin by studying the differentiability properties of the solution operator associated to the state equation \eqref{eq:controlproblem_2}.

\begin{remark}
  The function \(\mathsf{m}\) corresponds to the derivative of the convex function \(D \mapsto \frac12 \max \left(0,\abs{D} -g \right)^2\), defined for all \(D \in \mathbb{R}^{N\times N}\). Therefore, \(\mathsf{m}\) is monotone, i.e.,
  \begin{equation*}\label{eq:monotone}
    \left( \mathsf{m}(E)-\mathsf{m}(D) \right):(E-D)\geq 0.
  \end{equation*}
\end{remark}

For each \(u \in \Lvec\), the existence of a unique solution \(y \in Y\) to \eqref{eq:controlproblem_2} is guaranteed by \cite[Theorem 2.1]{de2014nonsmooth}. Therefore, we can define the solution operator as follows.

\begin{definition} The solution operator \(S:U\rightarrow Y\) is defined by \(u \mapsto S(u)=y\), where \(y\) is the unique solution to (\ref{eq:controlproblem_2}) associated with \(u\).
\end{definition}

The solution operator is, in addition, Lipschitz continuous. Indeed, taking $u_1, u_2 \in \Lvec$ it follows that the difference $y_1-y_2=S(u_1)-S(u_2)$ fulfills the equation:
\begin{equation*}\label{eq:convergenceofyn}
    \mu \int_{\Omega} \varepsilon (y_1-y_2):  \varepsilon v +\nu \int_{\Omega} (\mathsf{m}(\varepsilon y_1)-\mathsf{m}(\varepsilon y_2)):\varepsilon v=\langle u_1 -u_2, v \rangle_{Y',Y}, \quad \forall v \in Y.
\end{equation*}
Testing with $v=y_1-y_2$, and taking into account the monotonicity of \(\mathsf{m}\) and the Cauchy Schwarz inequality, it follows that
\begin{equation}\label{eq:monotonecauchy}
    \mu \norm{\varepsilon(y_1-y_2)}^2 \leq \norm{u_1 -u_2}_{Y'}\norm{y_1-y_2}_Y .
\end{equation}
Using Korn's inequality \cite[Thm.~6.15-1]{Ciarlet} in \eqref{eq:monotonecauchy}, we then obtain the existence of a constant \(C_K>0\) such that
\begin{equation}\label{eq:solopislipschitz}
    \norm{S(u_1)-S(u_2)}_Y \leq \dfrac{1}{\mu C_K}  \norm{u_1 -u_2}_{Y'}.
\end{equation}

It is also known from \cite{de2014nonsmooth} that, since $u \in \Lvec$, the solution has the additional regularity $y \in H^2_{loc}(\Omega)^N$. Moreover, if $N=2$ then the shear strain rate tensor satisfies the H\"older regularity $\varepsilon y \in C^{0,\alpha}(\Omega;\mathbb R^{N \times N})$, for $0<\alpha<1$.


\subsection{Directional differentiablity of the nonsmooth term}
In order to obtain the directional derivative of the Nemytskii operator associated with \(\mathsf{m}\), as a function from \(\Lmat\) to \(\Lmat\),
we will make use of the directional derivative of \(\mathsf{m}\) as a function from $\mathbb{R}^{N\times N}$ to $\mathbb{R}^{N\times N}$.

An equivalent way of writing the nonlinearity
\(\mathsf{m}:\mathbb{R}^{N\times N}  \rightarrow \mathbb{R}^{N\times N} \) is given by:
\begin{equation*}\label{eq:nonlinearity}
    \mathsf{m}(E)=\begin{cases}
        E -g \dfrac{E}{\abs{E}}, \quad &\text{if } \abs{E}> g\\
        0, \quad & \text{if }\abs{E}\leq g.
    \end{cases}
\end{equation*}
First, note that the function \(\mathsf{m}\) is continuous. In order to compute the directional derivative of \(\mathsf{m}\) in \(E \in \mathbb{R}^{N \times N}\), we consider three cases: \emph{i)} \(\abs{E}<g\), \emph{ii)} \(\abs{E}>g\) and \emph{iii)} \(\abs{E}=g\).

\textbf{Case 1.} Let \(E \in \mathbb{R}^{N \times N}\) be such that \(\abs{E}<g\).
Since the set \(\{E \in \mathbb{R}^{N \times N}:\abs{E}<g   \} \) is open, for all \(H \in \mathbb{R}^{N \times N} \) we have \[\abs{E +tH}<g, \]for \(t>0\) sufficiently small. Then, according to the definition of \(\mathsf{m}\), \(\mathsf{m}(E+tH)=0\) and \(\mathsf{m}(E)=0\).
Replacing these values in the definition of the directional derivative of \(\mathsf{m}\), we obtain
\begin{equation*}
  \underset{t\rightarrow 0}{\text{lim}} \dfrac{\mathsf{m}(E+tH)-\mathsf{m}(E)}{t}=0.
\end{equation*}
Thus,
\begin{equation*}
  \mathsf{m}'(E;H)=0, \quad \forall H \in \mathbb{R}^{N \times N}.
\end{equation*}

\textbf{Case 2.} Let \(E \in \mathbb{R}^{N \times N}\) be such that \(\abs{E}>g\).
Since the set \(\{E \in \mathbb{R}^{N \times N}:\abs{E}>g   \} \) is open as well, for all \(H \in \mathbb{R}^{N \times N} \) we have \[\abs{E +tH}>g, \]for \(t>0\) suficiently small. Replacing the corresponding values in the definition of the directional derivative of \(\mathsf{m}\), we get:
\begin{equation}\label{eq:directionalderm}
  \underset{t\rightarrow 0}{\text{lim}} \dfrac{\mathsf{m}(E+tH)-\mathsf{m}(E)}{t}=\underset{t\rightarrow 0}{\text{lim}}\dfrac{1}{t}\left\{tH-g\dfrac{E+tH}{\abs{E+tH}}+g\dfrac{E}{\abs{E}}\right\}=H-g\mathsf{F}'(E;
  H).
\end{equation}
where \(\mathsf{F}'(E;H)\) is the derivative of the function \(E \mapsto \mathsf{F}(E)=\frac{E}{\abs{E}}\) at \(E \in \mathbb{R}^{N\times N}\) in the direction \(H \in \mathbb{R}^{N\times N}\). Notice that \(\mathsf{F}\) is well-defined since \(\abs{E} >g\) and

\begin{equation} \label{eq:derivativeofF}
  \mathsf{F}'(E;H)=\dfrac{H}{\abs{E}}-\dfrac{E}{\abs{E}^3} (E:H).
\end{equation}
Inserting \eqref{eq:derivativeofF} in (\ref{eq:directionalderm}), we obtain
\begin{equation*}
  \mathsf{m}'(E;H)=H-g\left(\dfrac{H}{\abs{E}} -\dfrac{E}{\abs{E}^3} (E:H)\right), \quad \forall H\in \mathbb{R}^{N \times N}.
\end{equation*}

\textbf{Case 3.} Let \(E \in \mathbb{R}^{N \times N}\) be such that \(\abs{E}=g\). In this case,  \(\mathsf{m}(E)=0\) and depending on the direction \(H \in \mathbb{R}^{N\times N}\), for \(t>0\) sufficiently small, \(\mathsf{m}(E +tH)\) takes different values.
If \(H \in \mathbb{R}^{N\times N} \) is such that \(E: H \geq 0\), we have that \(\abs{E +tH}>g\) for \(t>0\) sufficiently small, and then
\begin{align*}
    \lim_{t\rightarrow 0} \dfrac{\mathsf{m}(E +tH)-\mathsf{m}(E )}{t}
    &=\lim_{t\rightarrow 0} \dfrac{E+tH}{\abs{E +tH}}\left(\dfrac{\abs{E +tH} -\abs{E}}{t} \right)=\dfrac{E}{\abs{E}^2}E:H.
\end{align*}

On the other hand, if \(H\) is such that \(E: H < 0\), for \(t>0\) sufficiently small, we have \(\abs{E+tH} < g\). In this case, we obtain
\begin{equation*}
  \lim_{t\rightarrow 0} \dfrac{\mathsf{m}(E +tH)-\mathsf{m}(E )}{t}=0.
\end{equation*}
In conclusion, if \(\abs{E}=g\), then we have
\begin{equation*}
  \mathsf{m}'(E;H)=
    \begin{cases}
    \dfrac{E}{g^2}(E:H), \quad & \text{if } E: H \geq 0,\\
    0,\quad &\text{if }E: H < 0.
    \end{cases}
\end{equation*}

In summary, the directional derivative of \(\mathsf{m}\) at \(E\) in direction \(H\) is given by
\begin{equation*}
  \mathsf{m}' (E;H)= \mathbbm{1}_{\{\abs{E}>g\} } \left( H+ \dfrac{gE}{\abs{E}^3} (E:H)-\dfrac{g H}{\abs{E}} \right) +\mathbbm{1}_{\{\abs{E}=g\} }\text{max}(0, (E:H))\dfrac{E}{g^2}.
\end{equation*}

\begin{theorem}
  Let \(G: L^2 \left( \Omega;\mathbb{R}^{N\times N}\right) \to L^2 \left( \Omega;\mathbb{R}^{N\times N}\right)\) be the Nemytskii operator associated with the nonlinearity \(\mathsf{m}\), i.e., for \(z \in L^2 \left( \Omega;\mathbb{R}^{N\times N}\right)\), \(G(z)(x):=\mathsf{m}(z(x))\), f.a.a. \(x \in \mathbb{R}^N\). Then \(G\) is directionally differentiable from \(L^2 \left( \Omega;\mathbb{R}^{N\times N}\right)\) to \(L^2 \left( \Omega;\mathbb{R}^{N\times N}\right)\).
\end{theorem}
\begin{proof}
  According to the definition of \(G\), for arbitrary \(x \in \mathbb{R}^N\) and \(z \in L^2 \left(\Omega;\mathbb{R}^{N\times N} \right) \), we have
  \begin{equation*}
   \underset{t\rightarrow 0}{\text{lim}}   \dfrac{G(z+th)(x)-G(z)( x)}{t}  =\underset{t\rightarrow 0}{\text{lim}}   \dfrac{\mathsf{m}(z (x)+ t h (x))-\mathsf{m}(z (x))}{t}.
  \end{equation*}
  Since \(\mathsf{m}\) is directionally differentiable, we have:
  \begin{equation*}
      \underset{t\rightarrow 0}{\text{lim}}   \dfrac{\mathsf{m}(z (x)+ t h (x))-\mathsf{m}(z (x))}{t}=\mathsf{m}'(z(x);h(x)), \quad \text{f.a.a. }x \in \Omega.
  \end{equation*}

  Moreover, since \(\mathsf{m}\) is Lipschitz continuous \cite[Section 2.2]{de2014nonsmooth} and  the directional derivative of \(\mathsf{m}'\) is globally Lipschitz continuous with respect to the direction \cite[Proposition 4.2.2]{luopang1996}, we have
  \begin{multline*}
      \abs{\dfrac{\mathsf{m}(z (x)+ t h (x))-\mathsf{m}(z (x))}{t}-\mathsf{m}'(z(x);h(x))}\\
      \leq \dfrac{\abs{z(x)+th(x)-z(x)}}{t}+\abs{h(x)}
      \leq 2\abs{h(x)},
  \end{multline*}
  a.e. in \(\Omega\).
  Hence, by Lebesgue Dominated Convergence Theorem, we obtain
  \begin{equation}\label{eq:nemytskiidifferentiable}
    \lim_{t\rightarrow 0} \left\|\dfrac{G(z+th)-G(z)}{t} -G'(z,h)\right\|=0,
  \end{equation}
  with
  \(G'(z;h)(x)=\mathsf{m}'(z(x);h(x))\) a.e. in \(\Omega\).
\end{proof}

\subsection{Differentiability of the solution operator}
In order to study the directional differentiablity of the solution operator \(S\), we will directly analyze the limit of the quotient \[\dfrac{S(u+th)-S(u)}{t}, \text{ for a given }h \in U. \]
To do so, in the following proposition we start by establishing the existence and uniqueness of solutions for an auxiliary partial differential equation, whose solution will act thereafter as candidate for the directional derivative of the solution operator.
\begin{theorem}\label{th:linearizedunique}
  Let \(y \in Y\) be a solution for equation \eqref{eq:controlproblem_2}. For every \(h \in \Lvec\),  the equation
  \begin{equation}\label{eq:linearizedeq}
    \mu \int_{\Omega} \varepsilon z : \varepsilon v  +\nu\int_{\Omega}   \mathsf{m}'\left( \varepsilon y;\varepsilon z \right)   : \varepsilon v = \int_{\Omega} h \cdot v,    \quad \forall v \in Y,
  \end{equation}
  has a unique solution \(z \in Y\). Here, \(\mathsf{m}'\left( \varepsilon y;\varepsilon z \right)\) denotes the directional derivative of the Nemytskii operator associated with the function \(\mathsf{m}\), at \(\varepsilon y\) in the direction \(\varepsilon z\).
\end{theorem}

\begin{proof}
  We will show that for every \(l \in Y'\), equation \(T(z)=l\), where  \(T:Y\rightarrow Y'\) is the operator defined by
  \[\langle T(z),v \rangle_{Y',Y}:=\mu \left( \varepsilon z,\varepsilon v \right) +\nu \left( \mathsf{m}'(\varepsilon y; \varepsilon z) , \varepsilon v\right), \forall v \in Y,\]
  has a unique solution. To this end, we will verify next that the operator \(T\) is hemicontinuous, monotone and coercive.

  First, we show that \(T\) is hemicontinuous. Let \((t_n)_{n \in \mathbb{N}}\) be a convergent sequence in \([0,1]\) with limit \(t \in [0,1]\). In the first place, if \(n \rightarrow \infty\), we have
  \begin{equation} \label{eq:hemicont1}
    \underset{n\rightarrow \infty}{\text{lim}}  \int_{\Omega} \varepsilon (z+t_n w): \varepsilon v = \int_{\Omega} \varepsilon (z+t w): \varepsilon v.
  \end{equation}

  Now, if \(\abs{\varepsilon y}< g\), \(\mathsf{m}'(\varepsilon y; \varepsilon (z+t_n w))\)=0, for every \(n \in \mathbb{N}\) and then \[\mathsf{m}'(\varepsilon y; \varepsilon (z+t_n w)) \rightarrow \mathsf{m}'(\varepsilon y; \varepsilon (z+t w)) =0 \quad \text{a.e. in } \Omega.\]
  By the Lebesgue Dominated Convergence Theorem, we obtain
  \begin{equation} \label{eq:hemicont2}
      \underset{n \rightarrow \infty }{\text{lim}} \int_{\Omega}\mathbbm{1}_{\{\abs{\varepsilon y }<g \}} \mathsf{m}'(\varepsilon y; \varepsilon (z+t_n w)) :\varepsilon v = \int_{\Omega}\mathbbm{1}_{\{\abs{\varepsilon y }<g \}}\mathsf{m}'(\varepsilon y;\varepsilon (z +tw)): \varepsilon v.
  \end{equation}
  If \(\abs{\varepsilon y}>g\), then \(\mathsf{m}'(\varepsilon y;\varepsilon (z +t_n w))=\varepsilon (z +t_n w)-g\left( -\frac{\varepsilon y}{\abs{\varepsilon y}^3} \varepsilon y:\varepsilon (z +t_n w)+\frac{\varepsilon (z +t_n w)}{\abs{\varepsilon y}}\right)\), for every \(n \in \mathbb{N}\). Hence, if \(n \rightarrow \infty\), we have that
  \begin{equation*}
      \mathsf{m}'(\varepsilon y;\varepsilon (z +t_n w)) \rightarrow \varepsilon (z +t w)-g\left( -\dfrac{\varepsilon y}{\abs{\varepsilon y}^3} \varepsilon y:\varepsilon (z +t w)+\dfrac{\varepsilon (z +t w)}{\abs{\varepsilon y}}\right)
  \end{equation*}
  a.e. in \(\Omega\). Moreover,
  \begin{multline*}
      \abs{\varepsilon (z +t_n w):\varepsilon v+g \dfrac{\varepsilon y:  \varepsilon v}{\abs{\varepsilon y}^3} \varepsilon y:\varepsilon (z +t_n w)-g\dfrac{\varepsilon (z +t_n w): \varepsilon v}{\abs{\varepsilon y}} }\leq \abs{\varepsilon (z+t_n w): \varepsilon v}\\+\dfrac{g}{\abs{\varepsilon y}^3 }\abs{\varepsilon y: \varepsilon v}\abs{\varepsilon y: \varepsilon (z+t_n w)}+\dfrac{g}{\abs{\varepsilon y}}\abs{\varepsilon(z+t_n w): \varepsilon v}
  \end{multline*}
  Therefore,
  \begin{equation*}
    \begin{aligned}
      \abs{\mathsf{m}'(\varepsilon y;\varepsilon (z +t_n w)): \varepsilon v} &\leq \left( \abs{\varepsilon z}+\abs{\varepsilon w} \right) \abs{\varepsilon v}+\dfrac{g \abs{\varepsilon y}^2}{\abs{\varepsilon y}^3} \abs{\varepsilon v}(\abs{\varepsilon z}+\abs{\varepsilon w}) +(\abs{\varepsilon z}+\abs{\varepsilon w})\abs{\varepsilon v}\\
      & \leq 3 (\abs{\varepsilon z}+\abs{\varepsilon w})\abs{\varepsilon v},
    \end{aligned}
  \end{equation*}
  and by the Lebesgue Dominated Converge Theorem, we get
  \begin{equation} \label{eq:hemicont3}
    \underset{n \rightarrow \infty }{\text{lim}} \int_{\Omega}\mathbbm{1}_{\{\abs{\varepsilon y }>g \}} \mathsf{m}'(\varepsilon y; \varepsilon (z+t_n w)) :\varepsilon v = \int_{\Omega}\mathbbm{1}_{\{\abs{\varepsilon y }>g \}}\mathsf{m}'(\varepsilon y;\varepsilon (z +tw)): \varepsilon v.
  \end{equation}
  Now, if \(\abs{\varepsilon y}=g\), for every \(n \in \mathbb{N}\)
  \begin{equation*}
  \mathsf{m}'(\varepsilon y;\varepsilon(z+t_n w))=
   \begin{cases}
   \dfrac{\varepsilon y}{g^2}\varepsilon y:\varepsilon(z+t_n w), \quad & \text{if }\varepsilon(z+t_n w):\varepsilon y \geq 0,\\
   0,\quad &\text{if }\varepsilon(z+t_n w):\varepsilon y<0.
   \end{cases}
  \end{equation*}
  Since, \(\varepsilon(z+t_n w):\varepsilon y \rightarrow \varepsilon(z+t w):\varepsilon y \) a.e. in $\Omega$, as \(n \rightarrow \infty\), we have,
  \begin{equation}
      \mathsf{m}'(\varepsilon y;\varepsilon(z+t_n w)) \rightarrow \mathsf{m}'(\varepsilon y;\varepsilon(z+t w)),
  \end{equation}
  a.e. in \(\Omega\), as \(n \rightarrow \infty\).
  If \(\varepsilon(z+t_n w): \varepsilon y <0\), we have that
  \begin{equation}\label{eq:boundedcase2}
    \abs{\mathsf{m}'(\varepsilon y;\varepsilon(z+t_n w)):\varepsilon v}\leq 0,
  \end{equation}
  for every \(n \in \mathbb{N}\).
  Also, if \(\varepsilon(z+t_n w):\varepsilon y \geq 0\), we have
  \begin{equation}\label{eq:boundedcase1}
    \begin{aligned}
      \abs{\mathsf{m}'(\varepsilon y;\varepsilon(z+t_n w)):\varepsilon v}&= \abs{\frac{\varepsilon y:\varepsilon v}{g^2}\varepsilon y:\varepsilon(z+t_n w)}\\
      &\leq \dfrac{\abs{\varepsilon y}^2}{g^2}\abs{\varepsilon v}(\abs{\varepsilon z}+t_n \abs{\varepsilon w} )\\
      &\leq \abs{\varepsilon v}(\abs{\varepsilon z}+ \abs{\varepsilon w} ).
    \end{aligned}
  \end{equation}
  for every \(n \in \mathbb{N}\).  According to (\ref{eq:boundedcase1}) and (\ref{eq:boundedcase2}), the function \(\abs{\mathsf{m}'(\varepsilon y;\varepsilon(z+t_n w))}\) is bounded in \(L^2 (\Omega)\), for every \(n \in \mathbb{N} \). Thus, thanks to the  Lebesgue Dominated Convergence Theorem, we conclude \begin{equation} \label{eq:hemicont4}
  \underset{n \rightarrow \infty }{\text{lim}} \int_{\Omega}\mathbbm{1}_{\{\abs{\varepsilon y }=g \}} \mathsf{m}'(\varepsilon y; \varepsilon (z+t_n w)) :\varepsilon v = \int_{\Omega}\mathbbm{1}_{\{\abs{\varepsilon y }=g \}}\mathsf{m}'(\varepsilon y;\varepsilon (z +tw)): \varepsilon v.
  \end{equation}
  Equations (\ref{eq:hemicont1}), (\ref{eq:hemicont2}), (\ref{eq:hemicont3}) and (\ref{eq:hemicont4}) imply that for every \(v,z,w \in Y\)
  \begin{equation*}
    \langle T(z+t_n w), v\rangle_{Y',Y}\rightarrow \langle T(z+t w), v\rangle_{Y',Y},
  \end{equation*}
  if \(n \rightarrow \infty\). Thus, the operator \(T\) is hemicontinuous.\\

  Nex, we will show that \(T\) is monotone. Since for any $z, y \in Y$,
  \begin{multline*}
    \langle T(z)-T(w),z-w\rangle_{Y',Y}=\mu \left( \varepsilon (z-w),\varepsilon (z-w) \right)\\
    +\nu \left( \mathsf{m}'(\varepsilon y; \varepsilon z) -\mathsf{m}'(\varepsilon y; \varepsilon w) , \varepsilon (z-w)\right),\quad \forall z,w \in Y,
  \end{multline*}
  we will focus on the second term on the right hand side, as for the first one the conclusion is straightforward. Let \(z,w \in Y\) be arbitrary, then
  \begin{multline*}
    \nu \left( \mathsf{m}'(\varepsilon y; \varepsilon z) -\mathsf{m}'(\varepsilon y; \varepsilon w) , \varepsilon (z-w)\right)=\nu \int_{\Omega} \mathbbm{1}_{ \{\abs{\varepsilon y}>g\} }\left[  \dfrac{g}{\abs{\varepsilon y}}\left(-\varepsilon(z-w):\varepsilon(z-w)\dfrac{}{}\right.\right.\\
    \left.
    \left.\dfrac{}{} +\dfrac{(\varepsilon y:\varepsilon(z-w))^2 }{\abs{\varepsilon y}^2}\right)+\varepsilon(z-w):\varepsilon(z-w)\right]+\nu \int_{\Omega}\mathbbm{1}_{ \{\abs{\varepsilon y}=g\} } (\mathsf{m}' (\varepsilon y;\varepsilon z)-\mathsf{m}'(\varepsilon y;\varepsilon w)):\varepsilon(z-w).
  \end{multline*}
  From the expression of $\mathsf m'$, it is easy to verify that
  \begin{equation}\label{eq:monotonyofM}
    (\mathsf{m}'(\varepsilon y;\varepsilon z)-\mathsf{m}'(\varepsilon y;\varepsilon w)):\varepsilon(z-w)\geq 0 \qquad \text{if } |\varepsilon y|=g.
  \end{equation}
  Now, if \(\abs{\varepsilon y}>g\), we have that \(1-\frac{g}{\abs{\varepsilon y}}>0\) and, consequently, \begin{equation}\label{eq:casemonotonicity1}
    \left(1-\dfrac{g}{\abs{ \varepsilon y}} \right) \varepsilon(z-w): \varepsilon(z-w)+\dfrac{(\varepsilon y: \varepsilon(z-w))^2}{\abs{\varepsilon y}^2}\geq 0
  \end{equation}
  Thanks to (\ref{eq:monotonyofM}) and (\ref{eq:casemonotonicity1}), we have that for every \(z,w \in Y\):
  \begin{equation}
    \left( \mathsf{m}'(\varepsilon y; \varepsilon z) -\mathsf{m}'(\varepsilon y; \varepsilon w) , \varepsilon (z-w)\right)\geq 0.
  \end{equation}
  Consequently, for every \(z,w \in Y\)
  \begin{align}
      \langle T(z)-T(w), z-w \rangle_{Y',Y}&\geq \mu \int_{\Omega}\varepsilon (z-w):\varepsilon (z-w) \geq \mu C_K \norm{z-w}_Y^2, \label{eq:monotonicityofT}
  \end{align}
  where \(C_K\) is a positive constant provided by the Korn's inequality. From (\ref{eq:monotonicityofT}), we conclude that \(T\) is a monotone operator. Moreover, choosing \(z\neq w\) in the same equation, we have that \(T\) is a strictly monotone operator.

  Further, note that in equation (\ref{eq:monotonicityofT}), for the particular case \(w=0\), we obtain that for every \(z \in Y\)
  \begin{equation*}
      \langle T(z) ,z \rangle_{Y',Y}\geq \mu C_K \norm{z}_Y^2.
  \end{equation*}
  Thanks to the Minty-Browder theorem, equation
  \begin{equation}\label{eq:equationforT}
    \langle T(z) ,v \rangle_{Y',Y} =  \langle l,v \rangle_{Y',Y},\quad \forall v \in Y,
  \end{equation}
  has then a unique solution \(z \in Y\). This holds, in particular, if \(l \in Y'\) is defined by
  \begin{equation*}
      \langle l,v\rangle_{Y',Y}=\int_\Omega h\cdot v \, dx, \quad \forall v \in Y,
  \end{equation*}
  with \(h\in \Lvec \).
  \end{proof}

\begin{theorem}\label{th:solutionopdirdiff}
  The solution operator \(S:U\rightarrow Y\) associated with the state equation (\ref{eq:controlproblem_2}) is directionally differentiable, and its directional derivative \(\delta=S'(u;h) \in Y\) at \(u \in U\), in the direction \(h \in U\), is given by the unique solution to
  \begin{equation}\label{eq:linearized}
    \mu \int_{\Omega} \varepsilon z : \varepsilon v +\nu\int_{\Omega}   \mathsf{m}'\left( \varepsilon y;\varepsilon z \right)   : \varepsilon v = \int_{\Omega} h \cdot v,    \quad \forall v \in Y,
  \end{equation}
  where \(y=S(u)\).
\end{theorem}

\begin{proof}
    We denote by \(y_t\) the solution to equation (\ref{eq:controlproblem_2}) corresponding to \(u+th\):
\begin{equation}\label{eq:equationyt}
     \mu \int_{\Omega} \varepsilon y_t : \varepsilon v +\nu\int_{\Omega}   \text{max}\left(0,\abs{\varepsilon y_t}  -g \right)  \, \dfrac{\varepsilon y_t}{\abs{\varepsilon y_t}}: \varepsilon v = \int_{\Omega} (u+th) \cdot v,    \quad \forall v \in Y.
\end{equation}
Substracting (\ref{eq:equationyt}) from (\ref{eq:controlproblem_2}), we obtain:
\begin{multline}\label{eq:ztequation}
    \mu \int_{\Omega} \varepsilon (y_t-y) : \varepsilon v +\nu\int_{\Omega} \left(  \text{max}\left(0,\abs{\varepsilon y_t}  -g \right)  \, \dfrac{\varepsilon y_t}{\abs{\varepsilon y_t}} \right.\\ \left. - \text{max}\left(0,\abs{\varepsilon y}  -g \right)  \, \dfrac{\varepsilon y}{\abs{\varepsilon y}} \right): \varepsilon v
     =t \int_{\Omega} h  \cdot v,    \quad \forall v \in Y.
\end{multline}

Let \(\overline{z} \in Y\) be the unique solution to (\ref{eq:linearizedeq}) and define \(z_t:=\frac{y_t-y}{t}\). Substracting (\ref{eq:linearizedeq}) from (\ref{eq:ztequation}), we get
\begin{equation*}
    \mu \int_{\Omega} \varepsilon z_t: \varepsilon v -\mu \int_{\Omega} \varepsilon \overline{z}:\varepsilon v +\nu \int_{\Omega} \dfrac{1}{t}\left( \mathsf{m}(\varepsilon y_t) -\mathsf{m}(\varepsilon y)\right): \varepsilon v -\nu \int_{\Omega} \mathsf{m}'(\varepsilon y;\varepsilon \overline{z} ) :\varepsilon v=0
\end{equation*}
for every \(v \in Y\).
Now, we rewrite this equation as:
\begin{multline*}
    \mu \int_{\Omega} \varepsilon (z_t-\overline{z}): \varepsilon v +\nu \int_{\Omega} \dfrac{1}{t}\left(  \mathsf{m}(\varepsilon y_t)-\mathsf{m}(\varepsilon y+t\varepsilon\overline{z})\right): \varepsilon v \\+\nu \int_{\Omega} \left(\dfrac{\mathsf{m}(\varepsilon y+t \varepsilon\overline{z}) -\mathsf{m}(\varepsilon y)}{t} -\mathsf{m}'(\varepsilon y;\varepsilon \overline{z}) \right):\varepsilon v=0, \quad \forall v \in Y.
\end{multline*}
Testing with \(v=z_t -\overline{z}\), we obtain:
  \begin{multline}\label{eq:differenceztz}
    \mu \int_{\Omega} \varepsilon (z_t-\overline{z}): \varepsilon ( z_t -\overline{z}) +\nu \int_{\Omega} \dfrac{1}{t}\left(  \mathsf{m}( \varepsilon y+t\varepsilon z_t)-\mathsf{m}(\varepsilon y+t\varepsilon\overline{z})\right): \varepsilon (z_t-\overline{z}) \\=-\nu \int_{\Omega} \left(\dfrac{\mathsf{m}(\varepsilon y+t \varepsilon\overline{z}) -\mathsf{m}(\varepsilon y)}{t} -\mathsf{m}'(\varepsilon y;\varepsilon \overline{z}) \right):\varepsilon (z_t -\overline{z}).
  \end{multline}
  Since the function \(\mathsf{m}\) is monotone,
  \begin{equation*}
   \left(  \mathsf{m}( \varepsilon y+t\varepsilon z_t)-\mathsf{m}(\varepsilon y+t\varepsilon\overline{z})\right): \varepsilon (z_t-\overline{z}) \geq 0.
  \end{equation*}
  Hence, (\ref{eq:differenceztz}) implies
  \begin{equation*}
    \mu \norm{\varepsilon(z_t -\overline{z})}^2 \leq -\nu \int_{\Omega} \left(\dfrac{\mathsf{m}(\varepsilon y+t \varepsilon\overline{z}) -\mathsf{m}(\varepsilon y)}{t} -\mathsf{m}'(\varepsilon y;\varepsilon \overline{z}) \right):\varepsilon (z_t -\overline{z}).
  \end{equation*}
  Using the Cauchy-Schwarz inequality we get that
  \begin{equation*}
    \mu \norm{\varepsilon(z_t -\overline{z})}^2 \leq \nu \left\| \dfrac{\mathsf{m}(\varepsilon y+t \varepsilon\overline{z}) -\mathsf{m}(\varepsilon y)}{t} -\mathsf{m}'(\varepsilon y;\varepsilon \overline{z})\right\| \norm{\varepsilon (z_t -\overline{z})},
  \end{equation*}
  which implies
  \begin{equation}\label{eq:boundedzt-z}
    \mu \norm{\varepsilon(z_t -\overline{z})} \leq \nu \left\| \dfrac{\mathsf{m}(\varepsilon y+t \varepsilon\overline{z}) -\mathsf{m}(\varepsilon y)}{t} -\mathsf{m}'(\varepsilon y;\varepsilon \overline{z})\right\|.
  \end{equation}
  Now, using the Korn's inequality in (\ref{eq:boundedzt-z}), we obtain the existence of a constant \(C_K>0\) such that
  \begin{equation}\label{eq:limitsequencezt}
    \mu C_K \norm{z_t -\overline{z}}_{Y} \leq \nu \left\|\dfrac{\mathsf{m}(\varepsilon y+t \varepsilon\overline{z}) -\mathsf{m}(\varepsilon y)}{t} -\mathsf{m}'(\varepsilon y;\varepsilon \overline{z})\right\|.
  \end{equation}
  Thanks to the directional differentiability of \(\mathsf{m}:\Lmat \rightarrow \Lmat\) at \(\varepsilon y\) in the direction of \(\varepsilon \overline{z}\), we know that, as \(t \rightarrow 0\),
  \begin{equation*}
     \left\|\dfrac{\mathsf{m}(\varepsilon y+t \varepsilon\overline{z}) -\mathsf{m}(\varepsilon y)}{t} -\mathsf{m}'(\varepsilon y;\varepsilon \overline{z})\right\|  \rightarrow 0.
  \end{equation*}
  In consequence, from (\ref{eq:limitsequencezt}), we obtain that \(\norm{z_t - \overline{z}}_Y \rightarrow 0\), which implies the result.
\end{proof}


\section{Existence of optimal controls and B-stationarity}
Next we analyze the existence of solutions for the optimal control problem \eqref{eq:controlproblem} and derive a first-order primal necessary optimality condition. To this aim, let us consider the reduced version of the optimal control problem, given by:
\begin{equation}\label{eq:reducedcontrolproblem}
  \min_{u \in U} j(u),
\end{equation}
where \(j:U \rightarrow \mathbb{R}\) is defined by \(j(u):=J(S(u),u)\). We recall here that the solution operator is uniquely determined through the solution of the state equation.

\begin{theorem}\label{th:existenceofglobalmin}
  There exists a global solution for the optimal control problem \eqref{eq:controlproblem}.
\end{theorem}
\begin{proof}
Since the functional \(J\) is bounded from below, there exists a sequence \( (y_n,u_n)_{n \in \mathbb{N}}\) in \(Y\times U\) such that
\begin{equation}\label{eq:minimizingseq}
  \lim_{n\rightarrow \infty} J(y_n,u_n)=\inf_{(y,u)\in Y\times U} J(y,u):=\hat{J}.
\end{equation}
Since
\begin{equation*}
  J(y_n,u_n)\geq \dfrac{\alpha}{2}\norm{u_n},
\end{equation*}
the  sequence \( (u_n)_{n \in \mathbb{N}}\) is bounded in $U$. Thanks to the Lipschitz continuity of the solution operator, the sequence \( (y_n)_{n \in \mathbb{N}}\) is bounded in $Y$ as well. Consequently,
there exists a constant \(C>0\) such that for every \(n \in \mathbb{N}\)
\begin{equation}
    \norm{(y_n ,u_n)}_{Y\times U}\leq C.
\end{equation}
Thus, there exist a subsequence, denoted in the same way, and limit points \(\overline{y} \in Y\) and \(\overline{u} \in U\) such that
\begin{equation}\label{eq:yuweaklyconverges}
  (y_n , u_n ) \rightharpoonup  (\overline{y}, \overline{u}) \quad \text{in }Y \times U.
\end{equation}
Due to the compact embedding \(U \hookrightarrow Y'\), (\ref{eq:yuweaklyconverges}) implies that \begin{equation}\label{eq:unconverges}
    u_n  \rightarrow \overline{u} \quad \text{in }Y'.
\end{equation}
Thanks to (\ref{eq:unconverges}) and the Lipschitz continuity of $S$, as \(n \rightarrow \infty \), it follows that \[y_n \rightarrow \overline y =S(\overline{u}) \quad \text{in } Y.\]
Since \(J\) is jointly convex and continuous, it is weakly lower semicontinuous, i.e.,
\begin{equation}
  J(\overline{y},\overline{u}) \leq \liminf_{n} J(y_n , u_n)=\hat{J}.
\end{equation}
In conclusion, \( (\overline{y}, \overline{u})\) is a global minimum of \eqref{eq:controlproblem}.
\end{proof}

A first primal optimality condition is directly obtained from the directional differentiability of the solution operator \(S\) and the smoothness of the cost functional. This type of optimality conditions is known in the literature as \emph{Bouligand stationarity}.

\begin{theorem}[B-stationarity]
  Let \((\overline{y},\overline{u}) \in Y \times U\) be a local minimum of problem \eqref{eq:controlproblem}. Then,
  \begin{equation}\label{eq:Bstationarity}
    J_y (\overline{y}, \overline{u})S'(\overline{u};h)+J_u(\overline{y}, \overline{u})h \geq 0, \quad \forall h \in U.
  \end{equation}
\end{theorem}

\begin{proof}
  Thanks to Proposition \ref{th:solutionopdirdiff} and \cite[Lema 3.9]{herzogmeyer2013}, the reduced cost functional \(j:U \rightarrow \mathbb{R}\) is directionally differentiable at \(\overline{u}\) and its directional derivative at \(\overline{u}\) in the direction of \(h\) is given by \begin{equation*}
    j'(\overline{u};h)=J'(\overline{y}, \overline{u})(S'(\overline{u};h),h)=J_y (\overline{y}, \overline{u})S'(\overline{u};h)+J_u (\overline{y}, \overline{u})h, \quad \forall h \in U.
    \end{equation*}
  Since \(\overline{u}\) is a local optimum of the problem
  \begin{equation}\label{eq:reducedProblemP}
    \min_{u \in U} j(u)
  \end{equation}and \(h\)
  is an admissible direction, we have that \(j'(\overline{u};h) \geq 0\) or, equivalently,
  \begin{equation*}
    J_y (\overline{y}, \overline{u})S'(\overline{u};h)+J_u (\overline{y}, \overline{u})h \geq 0, \quad \forall h \in U.
  \end{equation*}
\end{proof}


\section{Optimality conditions: Weak stationarity}
Similarly to general nonlinear programming, purely primal necessary conditions are not useful for characterizing local minima of nonsmooth optimal control problems, if they are not accompanied with Lagrange multiplier existence results. In the remaining of this section, existence of Lagrange multipliers is demonstrated by means of a regularization approach and a weak optimality system is obtained in the limit.

\subsection{Adapted regularization}
Let \( (\overline{y}, \overline{u} ) \in Y\times U\) be an arbitrary local minimum of problem \eqref{eq:controlproblem_1}-\eqref{eq:controlproblem_2}. We consider the following regularized optimal control problem:
\begin{equation}\tag{$P_\delta$} \label{eqn:regularizedprob}
  \begin{aligned}
  &\min_{(y,u)\in Y \times U}  \quad J_\delta(y,u)= \dfrac{1}{2} \, \norm{y-z_{d}}^{2}+\dfrac{\alpha}{2} \,\norm{u}^{2}+ \dfrac{1}{2} \norm{u-\overline{u}}^{2}\\
  \noindent & \text{subject to:} \\
  &\mu \int_{\Omega} \varepsilon y : \varepsilon v +\nu \int_{\Omega} \mathsf{m}_\delta (\varepsilon y) : \varepsilon v = \int_{\Omega} u \cdot v, \quad \forall v \in Y,
  \end{aligned}
\end{equation}
where \(\mathsf{m}_\delta\) is a regularized version of \(\mathsf{m}\) that satisfies the following assumptions.

\begin{assumption} \label{ass:properties of regularizing function}
  The regularized function \(\mathsf{m}_\delta\) satisfies the following:
  \begin{enumerate}
      \item \label{it:hipotesis1} \(\mathsf{m}_\delta \in C^1 \left( \mathbb{R}^{N\times N},\mathbb{R}^{N\times N} \right) \), for every \(\delta >0\).
      \item For every \(D \in \mathbb{R}^{N \times N}\)
      \begin{equation}\label{eq:convergenceofmdelta}
          \abs{\mathsf{m}_\delta (D)-\mathsf{m}(D) }\rightarrow 0, \quad \text{if } \delta \rightarrow 0,
      \end{equation}
        \item For every \(D , E \in \mathbb{R}^{N\times N} \)
      \begin{equation}\label{eq:regularizedmismonot}
      (\mathsf{m}_\delta (D)-\mathsf{m}_\delta (E): D - E) \geq 0.
      \end{equation}
      \item There exists a constant \(C>0\) such that for every \(D , E \in \mathbb{R}^{N\times N} \)
      \begin{equation} \abs{\mathsf{m}_\delta (D)-\mathsf{m}_\delta (E)}
      \leq C \abs{D -E}
      \end{equation}
      \item For every \( E \in \mathbb{R} ^{N\times N}\) and \( H \in \mathbb{R} ^{N\times N}\) we have that
      \begin{equation*} \label{th:mdeltaderivativemonotone}
          (\mathsf{m} _\delta ' (E)H,H)\geq 0.
      \end{equation*}
      \item There exists a constant \(C>0\), independent of \(E\), such that \[\abs{\mathsf{m}_\delta'(E) H } \leq C \abs{H}\]
      for every \(E, H \in \mathbb{R}^{N \times N}. \)\label{th:m_deltaboundedder}
      \item For every \(\epsilon >0\), the sequence \((\mathsf{m}_\delta'(E))_{\delta >0} \) converges  uniformly to $\mathsf{m}'(E)$, if \( \left|\abs{E}- g \right| \geq \epsilon \), as \(\delta \rightarrow 0\).
  \end{enumerate}
\end{assumption}

We can take for example the function \(\mathsf{m}_\delta:\mathbb{R}^{N\times N} \rightarrow \mathbb{R}^{N\times N}\) defined by \begin{equation*}
    \mathsf{m}_\delta (E)=\text{max}_\delta \left(0, \abs{E}-g \right)1_\delta (\abs{E}-g)\dfrac{E}{\abs{E}}, \quad \forall E \in \mathbb{R}^{N \times N},
\end{equation*}where, \(\text{max}_\delta (0, \cdot)\) is a regularized version of the  \(\text{max}(0, \cdot)\) function, defined by
\begin{equation*}
    \text{max}_\delta (0,x)=\begin{cases}
     x& \text{if }x >\delta ,\\
     -\dfrac{1}{16\delta^3 }x^4  +\dfrac{3}{8\delta}x^2 +\dfrac{1}{2}x +\dfrac{3\delta}{16} &\text{if } \abs{x}\leq \delta,\\
     0 &\text{if } x< -\delta
     \end{cases}
\end{equation*}for each  \(\delta >0\), and \(1_\delta(\cdot)\) is the derivative of the \(\text{max}_\delta (0, \cdot)\) function \cite[Section 4.2]{de2014nonsmooth}.

\begin{proposition} \label{th:linear_regularized}
  For every \(u \in \Lvec\), there exists a unique solution \(y_\delta \in Y\) to the state equation:
  \begin{equation}\label{eq:regularizedstateq}
    \mu \int_{\Omega} \varepsilon y_\delta : \varepsilon v +\nu \int_{\Omega} \mathsf{m}_\delta (\varepsilon y_\delta) : \varepsilon v = \int_{\Omega} u \cdot v,    \quad \forall v \in Y.
  \end{equation}
  Moreover, the solution operator of this equation, \( S_\delta:\Lvec \rightarrow Y\)
  is weakly continuous and G\^ateaux differentiable. The derivative of \(S_\delta\) at \(u \in \Lvec\) in the direction of \(h \in \Lvec\) is given by the unique solution \(w \in Y\) of the linearized equation:
  \begin{equation}\label{eq:regularizedlinearized}
    \mu \int_{\Omega} \varepsilon w : \varepsilon v +\nu \int_{\Omega} \mathsf{m}_\delta ' (\varepsilon y_\delta)\varepsilon w : \varepsilon v = \int_{\Omega} h \cdot v ,    \quad \forall v \in Y ,
  \end{equation}
  where \(y_\delta =S_\delta(u) \).
\end{proposition}

\begin{proof}
 Since \(\mathsf{m}_\delta \) is monotone and continuous, we have that the operator \(A:Y\rightarrow Y'\), defined by
 \begin{equation*}
     A(y)=\mu (\varepsilon y, \varepsilon \cdot)+\nu (\mathsf{m}_\delta (\varepsilon y),\varepsilon \cdot),
 \end{equation*}is strictly monotone, hemicontinuous and coercive. So, \eqref{eq:regularizedstateq} has a unique solution.
   Similarly to the proof of Proposition \ref{th:existenceofglobalmin}, we have that the operator \(S_\delta\) is weakly continuous. \\ Also, \eqref{eq:regularizedlinearized} has a unique solution \(w \in Y\) for every \(h \in \Lvec\), since \( \mathsf{m}_\delta\) satisfies Assumption \ref{ass:properties of regularizing function}.

   Let \(u, h \in \Lvec\) and \(t>0\) be arbitrary. Let \(y_{u+th}=S_\delta (u+th)\) and \(y_u =S_\delta (u)\). Substracting the equations satisfied by \(y_{u+th}\) and \(y_u\) we obtain \begin{equation}\label{eq:yuth}
       \mu (\varepsilon y_{u+th} -\varepsilon y_u , \varepsilon v )+\nu ( \mathsf{m}_\delta (\varepsilon y_{u+th})-\mathsf{m}_\delta (\varepsilon y_{u}), \varepsilon v)=(th,v), \quad \forall v \in Y.
   \end{equation}
   Now, we substract \eqref{eq:regularizedlinearized} from \eqref{eq:yuth}, and obtain \begin{equation*}
        \mu \int_\Omega \left( \dfrac{\varepsilon (y_ {u+th} - y_{u})}{t} -\varepsilon w \right):\varepsilon v  =-\nu \int_\Omega \left( \dfrac{\mathsf{m}_\delta (\varepsilon y_ {u+th})-\mathsf{m}_\delta (\varepsilon y_ {u})}{t}-\mathsf{m}_\delta' (\varepsilon y_u)\varepsilon w \right):\varepsilon v
    \end{equation*}for all \(v \in Y\).
    By defining \(w_t :=\frac{y_ {u+th} - y_{u}}{t}\), we get \(y_ {u+th}= y_u+t w_t\). Now we rewrite the last equation as
    \begin{multline*}
        \mu (\varepsilon (w_t-w),\varepsilon v)+\nu \left( \dfrac{\mathsf{m}_\delta (\varepsilon y_ {u}+t \varepsilon w_t) -\mathsf{m}_\delta (\varepsilon y_u + t \varepsilon w)}{t} , \varepsilon v \right)\\+\nu  \left(\dfrac{\mathsf{m}_\delta (\varepsilon y_u + t \varepsilon w)-\mathsf{m}_\delta (\varepsilon y_u )}{t} -\mathsf{m}_\delta' (\varepsilon y_u)\varepsilon w , \varepsilon v\right)=0, \quad \forall v \in Y.
    \end{multline*}
    Taking \(v=w_t -w\) in this equation, we get
    \begin{multline*}
      \mu \norm{\varepsilon (w_t -w) } + \dfrac{\nu}{t^2} \left( \mathsf{m}_\delta (\varepsilon y_ {u}+t \varepsilon w_t) -\mathsf{m}_\delta (\varepsilon y_u + t \varepsilon w) , t\varepsilon ( w_t - w) \right)\\
      \leq \nu \left\| \dfrac{\mathsf{m}_\delta (\varepsilon y_u + t \varepsilon w)-\mathsf{m}_\delta (\varepsilon y_u)}{t} - \mathsf{m}_\delta' (\varepsilon y_u) \varepsilon w \right \|  \norm{\varepsilon (w_t -w)}.
    \end{multline*}
    By \eqref{eq:regularizedmismonot}, we have
    \begin{equation*}
         \norm{\varepsilon (w_t -w) } \leq \dfrac{\nu}{\mu} \left\|\dfrac{\mathsf{m}_\delta (\varepsilon y_u + t \varepsilon w)-\mathsf{m}_\delta (\varepsilon y_u )}{t} -\mathsf{m}_\delta' (\varepsilon y_u)\varepsilon w\right\| \norm{\varepsilon (w_t -w)}.
    \end{equation*}
    Due to Korn's inequality, we obtain
    \begin{equation}\label{eq:wtconverges}
         C_K \norm{ w_t -w }_{Y} \leq \dfrac{\nu}{\mu} \left\| \dfrac{\mathsf{m}_\delta (\varepsilon y_u + t \varepsilon w)-\mathsf{m}_\delta (\varepsilon y_u )}{t} -\mathsf{m}_\delta' (\varepsilon y_u)\varepsilon w\right\| .
    \end{equation}
    Since the Nemytskii operator associated with \(\mathsf{m}_\delta\) is G\^ateaux differentiable from \(\Lmat\) to \( \Lmat \) \cite[Theorem 8]{goldberg1992nemytskij}, the right hand side of  \eqref{eq:wtconverges} converges towards zero as \(t \rightarrow 0\), and therefore, $w_t \rightarrow w \quad \text{in }Y.$
\end{proof}

\subsection{Consistency of the regularized solution operator}
Now, we study the convergence of the operator \(S_\delta \) to \(S\), as $\delta \to 0$.
\begin{proposition} \label{th:convofregularization}
 Let \(u \in  \Lvec \) be arbitrary but fixed. If \(\delta \rightarrow 0\), then
\begin{equation}\label{eq:convergenceofSdelta}
    S_\delta (u)\rightarrow S(u) \quad \text{in }Y.
\end{equation}
\end{proposition}

\begin{proof}
    Let \(u \in U\). We define \(y:=S(u)\) and \(y_\delta (u):=S_\delta (u)\). Substracting \eqref{eq:regularizedstateq} from  \eqref{eq:controlproblem_2}, we obtain that
    \begin{equation*}
        \mu \int_{\Omega} \varepsilon(y- y_\delta) : \varepsilon v +\nu \int_{\Omega}  (\mathsf{m}(\varepsilon y) -\mathsf{m}_\delta (\varepsilon y_\delta) ): \varepsilon v = 0,  \quad \forall v \in Y.
    \end{equation*}
    This equation is equivalent to
    \begin{multline}\label{eq:restofregularized}
        \mu \int_{\Omega} \varepsilon(y- y_\delta) : \varepsilon v +\nu \int_{\Omega}  (\mathsf{m}(\varepsilon y) -\mathsf{m} (\varepsilon y_\delta) ): \varepsilon v = \\ \nu \int_\Omega (\mathsf{m}_\delta (\varepsilon y_\delta)- \mathsf{m} (\varepsilon y_\delta) ):\varepsilon v  \quad \forall v \in Y.
    \end{multline}
    Testing with \(v=y-y_\delta \) we obtain that
    \begin{multline*}
         \mu \int_{\Omega} \varepsilon(y- y_\delta) : \varepsilon (y- y_\delta) +\nu \int_{\Omega}  (\mathsf{m}(\varepsilon y) -\mathsf{m} (\varepsilon y_\delta) ): (\varepsilon y- \varepsilon y_\delta) = \\ \nu \int_\Omega (\mathsf{m}_\delta (\varepsilon y_\delta)- \mathsf{m} (\varepsilon y_\delta) ):\varepsilon (y- y_\delta).
    \end{multline*}
    Thanks to the monotonicity of \(\mathsf{m}\) and the Cauchy-Schwarz inequality, we have that
    \begin{equation*}
         \norm{\varepsilon (y-y_\delta)} \leq \dfrac{\nu}{\mu} \norm{\mathsf{m}_\delta (\varepsilon y_\delta)- \mathsf{m} (\varepsilon y_\delta)}.
    \end{equation*}
    Now, Korn's inequality provides the existence of a positive constant \(C_K\) such that \begin{equation} \label{eq:convergenceofSdeltatoS}
        C_K\norm{y-y_\delta}_Y \leq \dfrac{\nu}{\mu} \norm{\mathsf{m}_\delta (\varepsilon y_\delta)- \mathsf{m} (\varepsilon y_\delta)}.
    \end{equation}
    Finally, \eqref{eq:convergenceofSdeltatoS} and \eqref{eq:convergenceofmdelta} imply \eqref{eq:convergenceofSdelta}.
\end{proof}

\begin{theorem}
  The regularized solution operator $S_\delta$ is Lipschitz continuous, i.e.,
  \begin{equation}\label{eq:SdeltaisLpischitz}
    \norm{S_\delta (u_1)-S_\delta (u_2)}_Y\leq \frac{1}{\mu C_K}\norm{u_1 -u_2 }_{Y'}.
  \end{equation}
  for arbitrary \(u_1 , u_2 \in Y'\).
\end{theorem}
\begin{proof}
    We define \(w_\delta :=S_\delta (u_1)\) and \(z_\delta :=S_\delta (u_2)\). Substracting the equations satisfied by \(w_\delta \) and \(z_\delta\), we obtain that:
    \begin{equation}\label{eq:restw_dz_d}
        \mu \int_\Omega \varepsilon (w_\delta - z_\delta ):\varepsilon v +\nu \int_\Omega (\mathsf{m}_\delta (\varepsilon w_\delta) - \mathsf{m}_\delta (\varepsilon z_\delta) ):\varepsilon v=\langle u_1 -u_2,  v\rangle, \quad \forall v \in Y.
    \end{equation}
    Taking \(v=w_\delta -z_\delta \) in \eqref{eq:restw_dz_d} and using the monotonicity of  \(m_\delta\), we get that \begin{equation*}
        \mu \norm{\varepsilon (w_\delta -z_\delta)}^2\leq \langle u_1 -u_2,  w_\delta -z_\delta \rangle .
    \end{equation*}
    Using the Cauchy-Schwarz inequality and the Korn inequality in the last equation, we arrive to \begin{equation*}\label{eq:lipschitzestimate}
        \mu C_K \norm{w_\delta -z_\delta }_Y^2 \leq \norm{u_1 -u_2}_{Y'} \norm{w_\delta -z_\delta}_Y.
    \end{equation*}
    which implies that
\begin{equation*}
     \norm{S_\delta (u_1) -S_\delta (u_2) }_Y \leq \dfrac{1}{\mu C_K}  \norm{u_1 -u_2}_{Y'} .
\end{equation*} 
\end{proof}


\begin{proposition}
  Let \((u_\delta)_{\delta>0} \) be a sequence in \(U\) such that \(u_\delta \rightharpoonup u\) in \(U\). Then, \(S_\delta( u_\delta) \to S(u)\) in Y, as \(\delta \to 0\).
\end{proposition}
\begin{proof}
  Since the embedding \(U \hookrightarrow Y'\) is compact, we get that \(u_\delta \rightarrow u\) in \(Y'\). Thanks to \eqref{eq:SdeltaisLpischitz}, we obtain that
  \begin{equation}\label{eq:convergenceduetoLipsch}
    S_\delta (u_\delta)\rightarrow S_\delta (u),
  \end{equation}as \(\delta \rightarrow 0\).
  The triangle inequality implies that, for every \(\delta >0\),
  \begin{equation*}
    \norm{S_\delta (u_\delta) -S(u) }_Y \leq \norm{S_\delta (u_\delta)-S_\delta (u)}_Y+\norm{S_\delta (u)-S(u) }_Y .
  \end{equation*}
  Then, from \eqref{eq:convergenceduetoLipsch} and \eqref{eq:convergenceofSdelta}, we have that \(S_\delta (u_\delta) \rightarrow S(u)\) in $Y$, as \(\delta \rightarrow 0\).
\end{proof}

\subsection{Regularized optimality system}
In order to derive an optimality condition for the regularized problem \eqref{eqn:regularizedprob}, let us start by defining the adjoint state as the solution $p_\delta \in Y$ of the equation:
\begin{equation*}
  \mu \int_{\Omega} \varepsilon p_\delta : \varepsilon v +\nu \int_{\Omega}[\mathsf{m}'_\delta (\varepsilon y_\delta )]^* \varepsilon p_\delta :  \varepsilon v = \int_{\Omega} (y_\delta -z_d) \cdot v, \quad \forall  v \in Y.
\end{equation*}
Thanks to the ellipticity of the terms on the left hand side, the equation is uniquely solvable.

\begin{theorem}
  Let \(\overline{u} \in U \) be a local solution of \eqref{eq:reducedProblemP} and \((y_\delta , u_\delta)\in Y \times U \) a local solution of \eqref{eqn:regularizedprob}. Let \(p_\delta \in Y\) be the adjoint state corresponding to \(y_\delta\). Then, the triplet \( (u_\delta , y_\delta , p_\delta) \) satisfies the following optimality system:
  \begin{align}
    \mu \int_{\Omega} \varepsilon y_\delta : \varepsilon v +\nu \int_{\Omega} \mathsf{m}_\delta (\varepsilon y_\delta ) : \varepsilon v &= \int_{\Omega} u_\delta \cdot v, &&\forall  v \in Y  \label{eq:statereg}\\
    \mu \int_{\Omega} \varepsilon p_\delta : \varepsilon v +\nu \int_{\Omega}[\mathsf{m}'_\delta (\varepsilon y_\delta )]^* \varepsilon p_\delta :  \varepsilon v &= \int_{\Omega} (y_\delta -z_d) \cdot v, && \forall  v \in Y  \label{eq:adjointeq}\\
    -p_\delta &= (\alpha+1) u_\delta -\overline{u}, && \text{a.e. in } \Omega \label{eq:gradienteq}
  \end{align}
\end{theorem}

\begin{proof}
    For every \(u \in U\), we define the reduced regularized cost functional \[j_\delta (u)=\frac{1}{2} \,   \norm{S_\delta (u) -z_{d}}  _{L^{2} (\Omega)^{N}}^{2}+\frac{\alpha}{2} \,\norm{u} _{L^{2} (\Omega)^{N} } ^{2}+\frac{1}{2} \norm{u-\overline{u}} _{L^{2} (\Omega)^{N} } ^{2} .\]
    Note that \(u_\delta\) is a local solution of the reduced problem
    \begin{equation}\label{eq:reducedregproblem}
        \min_{u\in U}\, j_\delta (u).
    \end{equation}
    Thanks to the chain rule, we have that \(j_\delta\) is 
    G\^ateaux differentiable and its derivative at \(u_\delta\) in the direction of \(h \in U\)  is given by
    \begin{align}
    j_\delta ' (u_\delta)h &=J_y (y_\delta, u_\delta )S_\delta '(u_\delta;h)+J_u (y_\delta, u_\delta)h+(u-\overline{u},h) \nonumber\\
    &=(y_\delta -z_d ,S_\delta '(u_\delta;h)) + \alpha (u,h)+(u-\overline{u},h) \label{eq:dirderivativejdelta}
    \end{align}
    Replacing \eqref{eq:adjointeq} in \eqref{eq:dirderivativejdelta}, we have that \begin{align}
         j_\delta ' (u_\delta)h & = \mu ( \varepsilon p_\delta , \varepsilon S_\delta '(u_\delta;h) ) +\nu ([\mathsf{m}'_\delta (\varepsilon y_\delta )]^* \varepsilon p_\delta ,  \varepsilon S_\delta '(u_\delta;h) )  + \alpha (u,h)+(u-\overline{u},h)\nonumber\\
         & =  \mu ( \varepsilon p_\delta , \varepsilon S_\delta '(u_\delta;h) ) +\nu ( \varepsilon p_\delta , \mathsf{m}'_\delta (\varepsilon y_\delta )  \varepsilon S_\delta '(u_\delta;h) )  + \alpha (u,h)+(u-\overline{u},h) \label{eq:replaceeyforeu} .
    \end{align}
    From the linearized equation \eqref{eq:regularizedlinearized} we know that
    \begin{equation*}
        \mu ( \varepsilon S_\delta '(u_\delta;h) , \varepsilon v)+\nu ( \mathsf{m}'_\delta (\varepsilon y_\delta )  \varepsilon S_\delta '(u_\delta;h) ,\varepsilon v   ) =(h, v).
    \end{equation*}
    Replacing the latter in \eqref{eq:replaceeyforeu}, we obtain that
    \begin{equation} \label{eq:derivativeofjdeltau}
        j_\delta ' (u_\delta)h= (p_\delta , h)+\alpha (u,h)+(u-\overline{u},h), \quad \forall h \in U.
    \end{equation}
    Since \(u_\delta\) is a local minimum of \eqref{eq:reducedregproblem}, we have that \(j_\delta ' (u_\delta)h=0\), for every \(h \in U\). Then from \eqref{eq:derivativeofjdeltau}, we deduce \eqref{eq:gradienteq}.
    \end{proof}

\begin{theorem} \label{th:convergenceofminimizers}
  Let \(\overline{u}\in U\) a local minimizer of the reduced problem \eqref{eq:reducedcontrolproblem}. Then there exists a sequence of local minimizers \((u_\delta)_{\delta >0}\) of problem \eqref{eqn:regularizedprob} such that
  \begin{equation} \label{eq:udeltaconvergence}
    u_\delta \rightarrow \overline{u} \quad \text{in } U,
  \end{equation}
  as \(\delta \rightarrow 0\). Moreover,
  \begin{equation}\label{eq:statesdeltaconverg}
    S_\delta (u_\delta) \rightarrow S(\overline{u}) \quad \text{in }Y.
  \end{equation}
\end{theorem}
\begin{proof}
    Let \(B(\overline{u},r)\) be a closed ball around \(\overline{u}\) such that:
    \begin{equation*}
        j(\overline{u})\leq j(u), \quad \forall u \in B(\overline{u},r).
    \end{equation*}
    We consider the following auxiliary optimal control problem:
    \begin{equation}\label{eq:auxcontrol}
      \min_{u \in B(\overline{u} ,r)} j_\delta (u)
    \end{equation}
The existence of global minimizers for the problem \eqref{eq:auxcontrol} can be proved similarly to the proof of Proposition \ref{th:existenceofglobalmin}. Let \(u_\delta\) be a global minimizer of \eqref{eq:auxcontrol}. Since for every \(\delta >0\), \(u_\delta \in B(\overline{u},r)\), there exists a subsequence, denoted by the same symbol, and a limit \(\hat{u} \in U\) such that:
\begin{equation}\label{eq:u_deltaweaklyconverges}
    u_\delta \rightharpoonup \hat{u} \quad \text{in } U.
\end{equation}
Since \(\overline{u}\) is feasible for problem \eqref{eq:auxcontrol}, we have that
\begin{equation}\label{eq:JSdeltaudelta}
    J(S_\delta (u_\delta), u_\delta ) \leq j_\delta (u_\delta)\leq  j_\delta (\overline{u})=J(S_\delta(\overline{u}),\overline{u}).
\end{equation}
Proposition \ref{th:convofregularization} implies that \(S_\delta (\overline{u}) \rightarrow S(\overline{u})\), as \(\delta \rightarrow 0\). Since \(J\) is continuous, we obtain that
\begin{equation*}
    J(S_\delta (\overline{u}), \overline{u} )\rightarrow J(S (\overline{u}), \overline{u}),
\end{equation*}
as \(\delta \rightarrow 0\). Then by \eqref{eq:JSdeltaudelta}, we obtain
\begin{equation}\label{eq:limsupjdelta}
   \underset{\delta \rightarrow 0}{\text{lim sup }} j_\delta( u_\delta ) \leq j(\overline{u}).
\end{equation}
On the other hand, from \eqref{eq:u_deltaweaklyconverges} and the properties of $J$, we get that
\begin{equation*}
  J(S(\hat{u}), \hat{u})+\dfrac{1}{2}\norm{\hat{u}-\overline{u}}^2 \leq \underset{\delta \rightarrow 0}{\text{lim inf }}j_\delta( u_\delta ).
\end{equation*}
Since  \(\overline{u}\) is a local minimizer of \eqref{eq:reducedProblemP}, it follows that \(j(\overline{u})\leq J(S(\hat{u}), \hat{u}) \). Then from the last equation we obtain that:
\begin{equation}\label{eq:liminfjdelta}
  j(\overline{u})+\dfrac{1}{2}\norm{\hat{u}-\overline{u}}^2 \leq \underset{\delta \rightarrow 0}{\text{lim inf }}j_\delta( u_\delta ).
\end{equation}
Combining \eqref{eq:limsupjdelta} and \eqref{eq:liminfjdelta}, we have that \begin{equation}\label{eq:limofjdeltaudelta}
  j(\overline{u})\leq  \underset{\delta \rightarrow 0}{\text{lim inf }}j_\delta( u_\delta ) \leq    \underset{\delta \rightarrow 0}{\text{lim sup }} j_\delta( u_\delta )\leq j(\overline{u}).
\end{equation}
Also,
\begin{equation*}
j(\overline{u})+\dfrac{1}{2}\norm{\hat{u}-\overline{u}}^2 \leq j(\overline{u}),
\end{equation*}
Hence, we obtain that \(\hat{u}=\overline{u}\).

Similarly, from \eqref{eq:JSdeltaudelta}-\eqref{eq:limsupjdelta}, we have that
\begin{equation*}
  \underset{\delta \rightarrow 0}{\text{lim sup }} J( S_\delta
  (u_\delta),u_\delta ) \leq j(\overline{u}).
\end{equation*}
Also,  \eqref{eq:u_deltaweaklyconverges} implies that
\begin{equation*}
  J(S(\hat{u}), \hat{u}) \leq \underset{\delta \rightarrow 0}{\text{lim inf }}J( S_\delta (u_\delta) , u_\delta ).
\end{equation*}
Combining the last equation and \eqref{eq:limsupjdelta}, we obtain that
\begin{equation*}
    j(\overline{u}) \leq J(S(\hat{u}), \hat{u}) \leq \underset{\delta \rightarrow 0}{\text{lim inf }}J( S_\delta (u_\delta) , u_\delta )\leq  \underset{\delta \rightarrow 0}{\text{lim sup }} J( S_\delta
        (u_\delta),u_\delta )  \leq j(\overline{u}).
\end{equation*}
From this equation, we have that \(J(S_\delta(u_\delta), u_\delta )\rightarrow j(\overline{u})\). From  \eqref{eq:limofjdeltaudelta}, we obtain that \(j_\delta (u_\delta)\rightarrow j(\overline{u})\), as \(\delta \rightarrow 0\). Since
\begin{equation*}
\dfrac{1}{2}\norm{u_\delta -\overline{u}}^2 = j_\delta (u_\delta)-J(S_\delta(u_\delta),u_\delta ).
\end{equation*}
we deduce that
\begin{equation}
    \dfrac{1}{2}\norm{u_\delta -\overline{u}}_U^2 \rightarrow 0,
\end{equation}as \(\delta \rightarrow 0\). Therefore, we conclude \eqref{eq:udeltaconvergence}. The convergence of the associated states of  \eqref{eq:statesdeltaconverg}, is deduced from Proposition \ref{th:convofregularization}.

Next, we will use a localization argument from \cite{casas2002error} in order to show that, for \(\delta >0\) sufficiently small, \(u_\delta\) is a local minimizer of \eqref{eq:reducedregproblem}. Let us take \(R=\frac{r}{2}\). Let \(u \in B(u_\delta, R)\) be arbitrary, then
\begin{equation}\label{eq:triangle1}
  \norm{u-u_\delta}_U \leq \dfrac{r}{2}.
\end{equation}
From \eqref{eq:udeltaconvergence}, there exists \(\overline{\delta}>0\), such that for every \(\delta \leq \overline{\delta}\), \begin{equation}\label{eq:triangle2}
     \norm{u_\delta -\overline{u}}_U \leq \dfrac{r}{2} .
\end{equation}Therefore, from the triangle inequality, we have
\begin{equation*}
    \norm{u-\overline{u}}_U \leq \norm{u-u_\delta}_U +\norm{u_\delta -\overline{u}}_U \leq \dfrac{r}{2}+\dfrac{r}{2} =r,
 \end{equation*}
 where we have used \eqref{eq:triangle1}
 and \eqref{eq:triangle2}. Consequently, we have that \(u\in B(\overline{u},r ) \). Since \(u_\delta \), is a global minimizer of \eqref{eq:auxcontrol}, with \(\delta \leq \overline{\delta}\),  we conclude that
 \begin{multline*}
      \frac{1}{2} \,   \norm{S_\delta (u_\delta) -z_{d}}  _{L^{2} (\Omega)^{N}}^{2}+\frac{\alpha}{2} \,\norm{u_\delta} _{L^{2} (\Omega)^{N} } ^{2}+\frac{1}{2} \norm{u_\delta-\overline{u}} _{L^{2} (\Omega)^{N} } ^{2} \leq  \frac{1}{2} \,   \norm{S_\delta (u) -z_{d}}  _{L^{2} (\Omega)^{N}}^{2}\\+\frac{\alpha}{2} \,\norm{u} _{L^{2} (\Omega)^{N} } ^{2}+\frac{1}{2} \norm{u-\overline{u}} _{L^{2} (\Omega)^{N} } ^{2} ,
 \end{multline*}
 for every \(u \in B(u_\delta , R ) \).
\end{proof}

\subsection{Passing to the limit}
In the following lemma we establish the boundedness of the dual variables as a preparatory step to derive a weak optimality system.
\begin{lemma} \label{th:dualvarbounded}
  Let \((p_\delta)_{\delta >0} \) the sequence of adjoint states defined by \eqref{eq:adjointeq} and associated with a sequence of local solutions \((u_\delta)_{\delta >0} \) from Proposition \ref{th:convergenceofminimizers}. For every \(\delta >0\), we define the Lagrange multiplier
  \begin{equation*}
      \lambda_\delta :=[\mathsf{m}'_\delta (\varepsilon y_\delta )]^* \varepsilon p_\delta \in \Lmat.
  \end{equation*}
  Then there exist constants \(C_1 ,C_2 >0\) such that, for all \(\delta >0\), we have
  \begin{equation*}
      \norm{p_\delta}_Y\leq C_1,  \quad \norm{\lambda_\delta} \leq C_2.
  \end{equation*}
\end{lemma}
\begin{proof}
  Since \((\lambda_\delta, \varepsilon v)=(\varepsilon p_\delta,\mathsf{m}'_\delta (\varepsilon y_\delta  ) \varepsilon v )\), testing with \(v=p_\delta\) in \eqref{eq:adjointeq} and using Assumption \ref{ass:properties of regularizing function}, we get
  \begin{equation*}
    \mu \norm{\varepsilon p_\delta}^2 \leq (y_\delta -z_d , p_\delta).
  \end{equation*}
  Using Korn's inequality and the Cauchy-Schwarz inequality, we get that
  \begin{equation*}
       \mu C_K \norm{ p_\delta}_{Y}^2 \leq \norm{y_\delta -z_d} \norm{ p_\delta},
  \end{equation*}
  Furthermore, the embedding \(Y \hookrightarrow \Lvec\) implies the existence of a constant \(C>0\) such that \begin{equation*}
      \mu C_K \norm{ p_\delta}_{Y} \leq C( C \norm{y_\delta}_Y +\norm{z_d}.
  \end{equation*}
  From \eqref{eq:statesdeltaconverg}, we deduce that the sequence \((y_\delta)_{\delta >0}\) is uniformly bounded, and, thus, the sequence \((p_\delta)_{\delta >0}\) is also bounded in $Y$.

  Now, we show that the sequence \((\lambda_\delta)_{\delta>0}\) is bounded. Since by Assumption \ref{ass:properties of regularizing function} , $|\mathsf{m}'_\delta (E)D | \leq C |D|,$ for all $D, E \in \mathbb{R}^{N \times N}$, it follows that
  \begin{equation*}
    \|\lambda_\delta\|  = \| \mathsf{m}'_\delta (\varepsilon y_\delta ) \varepsilon p_\delta \|  \leq  \| \mathsf{m}'_\delta (\varepsilon y_\delta )\|_{L^\infty} \|\varepsilon p_\delta \| ,
  \end{equation*}
  which, thanks to the boundedness of \((p_\delta)_{\delta >0}\) in $Y$, implies that the sequence \((\lambda_\delta)_{\delta >0}\) is bounded in $\Lmat$.
\end{proof}

In the following proposition we obtain an optimality system for the problem \eqref{eq:controlproblem}.
\begin{theorem}[Weak stationarity]\label{th:optimalityafterlimit}
  Let $(\overline{y},\overline{u}) \in Y \times U$ be a local solution to problem \eqref{eq:controlproblem}. There exist an adjoint state $p \in Y$ and a multiplier $\lambda \in \Lmat$, such that
  \begin{subequations} \label{eq: weak system}
	  \begin{align}
	    & \mu \int_{\Omega} \varepsilon \overline{y}  : \varepsilon v +\nu \int_{\Omega}  \mathsf{m}(\varepsilon \overline{y}): \varepsilon v = \int_{\Omega} \overline{u}  \cdot v,   && \forall v \in Y \label{eq:stateqcontrol}\\
	    & \mu \int_{\Omega} \varepsilon p : \varepsilon v +\int_{\Omega} \lambda : \varepsilon v = \int_{\Omega} (\overline{y} -z_d) \cdot v, && \forall v \in Y  \label{eq:adjequation}\\
	    & \alpha \overline{u} + p =0, && \text{a.e. in } \Omega \label{eq:gradequation}\\
			& \lambda  =0, && \text{a.e. in } \{|\varepsilon \bar y| <g \}\\
	    & \lambda  = \left(\varepsilon p+ \frac{g}{\abs{\varepsilon \overline{y}}^3} \left( \varepsilon \overline{y} :\varepsilon p \right) \, \varepsilon \overline{y}
	    -g \frac{\varepsilon p }{|\varepsilon \overline{y}|} \right),
			&& \text{a.e. in} \, \{ |\varepsilon \overline{y}| > g \}.
	  \end{align}
	\end{subequations}
\end{theorem}
\begin{proof}
  Let \((u_\delta)_{\delta >0}\) be the sequence from Theorem \ref{th:convergenceofminimizers} and let \((y_\delta)_{\delta>0}\) be the corresponding sequence of states. Since \(Y\) and \(\Lmat\) are reflexive, by Lemma \ref{th:dualvarbounded}, we obtain the existence of subsequences of \((p_\delta)_{\delta>0}\) and
  \( (\lambda_{\delta})_{\delta>0} \), denoted by the same symbol, and limit points \(p \in Y\) and \(\lambda \in \Lmat \), such that
  \begin{align}
    & p_\delta \rightharpoonup p \quad \text{in } Y,\\
    & \varepsilon p_\delta \rightharpoonup \varepsilon p \quad \text{in } \Lvec, \label{eq:weak convergence epdelta}\\
    & \lambda_\delta \rightharpoonup \lambda \quad \text{in } \Lmat.
  \end{align}
  Passing to the limit in \eqref{eq:adjointeq}, we then obtain \eqref{eq:adjequation}.

 From Theorem \ref{th:convergenceofminimizers}, $\varepsilon y_\delta \to \varepsilon \bar y$ strongly in $\Lmat$, as $\delta \to 0,$ which implies existence of a pointwise convergent subsequence. On the set $\mathcal I:=\{ |\varepsilon y| \neq g \}$, we also get that $\partial_C \mathsf{m}(\varepsilon \overline y)$ is a singleton. Consequently, from Assumption \ref{ass:properties of regularizing function}(7) and the pointwise convergence
 $\varepsilon y_\delta(x) \to \varepsilon \bar y(x)$ a.e. in $\Omega$, that $\mathsf{m}_\delta' (\varepsilon y_\delta(x)) \to \mathsf{m}' (\varepsilon \overline{y}(x))$ pointwise
 a.e. in $\mathcal I$.

	Thanks to Assumption \ref{ass:properties of regularizing function}(6) we may apply Lebesque dominated convergence theorem and get that
	\begin{equation}
		\mathsf{m}_\delta' (\varepsilon y_\delta) \to \mathsf{m}' (\varepsilon \overline{y}) \quad \text{in }\mathbb L^2(\mathcal I).
	\end{equation}
	Together with \eqref{eq:weak convergence epdelta} we then obtain that
	\begin{equation}
		\mathsf{m}_\delta' (\varepsilon y_\delta) \varepsilon p_\delta \rightharpoonup \mathsf{m}' (\varepsilon \overline{y}) \varepsilon p \quad \text{in }\mathbb L^2(\mathcal I),
	\end{equation}
	hence $\lambda = \mathsf{m}' (\varepsilon \overline{y}) \varepsilon p$ a.e. in $\mathcal I$.


  From the convergence of the sequences \((p_\delta)_{\delta >0}\) and \((u_\delta)_{\delta >0}\) and the compact embedding \(Y \hookrightarrow \Lvec\), we deduce \eqref{eq:gradequation}.

\end{proof}

Since the optimality system given by \eqref{eq:stateqcontrol}-\eqref{eq:gradequation} does not contain any information about the sign of \(p\) or \(\lambda\), this system can be considered as a weak stationarity condition.


\section{Strong stationarity}
The following lemma will be used in the proof of the subsequent theorems.
\begin{lemma}\label{th:densedirderivative}
Let \((\overline{y}, \overline{u})\in Y \times U\) be a local optimal solution of problem \eqref{eq:controlproblem}. The set \(\left\{ S'(\overline{u};h): h \in U\right\}\) is dense in \(Y\).
\end{lemma}
\begin{proof}
  Let \(\xi \in Y\) be arbitrary. We will show that there exists a sequence \((h_n)_{n \in N}\) in \(U\) such that
  \begin{equation*}
    S'(\overline{u};h_n)\rightarrow \xi \quad \text{in }Y.
  \end{equation*}
  Let \(\hat h \in Y'\) be the functional corresponding to the evaluation of the ``linearized'' equation   \eqref{eq:linearizedeq} at \(\xi\), i.e.,
  \begin{equation} \label{eq:right hand side of linearizedeq}
    \langle \hat h, v \rangle_{Y',Y}: = \mu \int_{\Omega} \varepsilon \xi : \varepsilon v +\nu\int_{\Omega}   \mathsf{m}'\left( \varepsilon \overline{y};\varepsilon \xi \right)   : \varepsilon v, \quad \forall v \in Y.
  \end{equation}
  As $Y \hookrightarrow H \hookrightarrow Y'$ densely, with $H:=\{ \operatorname{curl} \varphi: \varphi \in \Hvec, ~ \operatorname{curl} \varphi \cdot \vec n |_\gamma =0 \}$ (see, e.g., \cite[Chapter 3]{giraultraviart}), there exists a sequence \((h_n)_{n \in N} \subset U\) such that $h_n \to \hat h$ in $Y'$, as $n \to \infty$.

  Taking the difference between \eqref{eq:right hand side of linearizedeq} and the ``linearized'' equation satisfied by each $S'(\overline{u};h_n)$, and owing to the monotonicity of the ``linearized'' operator (see Theorem \ref{th:linearizedunique}), we obtain that
  \begin{equation*}
    \|S'(\overline{u};h_n)- \xi \|_Y \leq C \|h_n - \hat h\|_{Y'},
  \end{equation*}
  which implies the result.
\end{proof}

\begin{theorem}\label{th:strongstationarity}
 Let \((\overline{y}, \overline{u})\in Y \times U\) a local optimal solution of \eqref{eq:controlproblem}. Then there exist a unique adjoint state \(p\in Y\) and a multiplier \(\lambda \in \Lmat \) such that
 \begin{subequations} \label{eq:strongOS}
   \begin{align}
    & \mu (\varepsilon p, \varepsilon v) + \nu \int_{\Omega} \lambda :  \varepsilon v = (\overline{y}-z_d , v), &&  \forall v \in Y, \label{eq:stronglinearizedy} \\
    & p+\alpha \overline{u} =0, && \text{a.e. in } \Omega.\label{eq:stronglinearizedu}\\
		& \lambda  =0, && \text{a.e. in } \{|\varepsilon \bar y| <g \} \label{eq:stronglinearizedlambda1}\\
		& \lambda  = \left(\varepsilon p+ \frac{g}{\abs{\varepsilon \overline{y}}^3} \left( \varepsilon \overline{y} :\varepsilon p\right) \, \varepsilon \overline{y}
		-g\dfrac{\varepsilon p }{\abs{\varepsilon \overline{y} }} \right), && \text{a.e. in }\{|\varepsilon \bar y| >g \},\label{eq:stronglinearizedlambda2}\\
		& \int_{\mathcal A} \lambda :  \varepsilon v \geq \int_{\mathcal A} \frac{1}{g^2} \max(0,(\varepsilon \overline{y}: \varepsilon v)) \varepsilon \overline{y}:\varepsilon p,
		&& \forall v \in Y, \label{eq:stronglinearizedlambda3}
  \end{align}
\end{subequations}
where $\mathcal A := \{x \in \Omega: \abs{\varepsilon \overline{y}(x)} = g\}$.
\end{theorem}

\begin{proof}
  Let \(p \in Y\) and \(\lambda \in \Lmat \) satisfy \eqref{eq: weak system} and let \(h \in U\) be arbitrary. Equations \eqref{eq:stronglinearizedlambda1} and \eqref{eq:stronglinearizedlambda2} follow directly from the weak optimality system \eqref{eq: weak system}. Taking \(v=S'(\overline{u},h) \in Y\) in \eqref{eq:adjequation}, we obtain
  \begin{multline}\label{eq:JyinderivativeofS}
    \mu \int_{\Omega} \varepsilon p : \varepsilon S'(\overline{u},h) +\nu \int_{\Omega} \lambda : \varepsilon S'(\overline{u},h) \\ = \int_{\Omega} (\overline{y} -z_d) \cdot S'(\overline{u},h) = J_y (\overline{y}, \overline{u})S'(\overline{u},h).
  \end{multline}
  On the other hand, taking \(v=p\) in equation \eqref{eq:linearizedeq}, we have
  \begin{equation}\label{eq:linearizedtestp}
  \mu \int_{\Omega} \varepsilon S'(\overline{u},h) : \varepsilon p +\nu\int_{\Omega}   \mathsf{m}'\left( \varepsilon \overline{y};\varepsilon S'(\overline{u},h) \right)   : \varepsilon p = \int_{\Omega} h \cdot p.
  \end{equation}
  From Proposition \ref{th:optimalityafterlimit}, we know that $p=- \alpha \overline{u}, \text{ a.e. in }\Omega.$
  Replacing this in \eqref{eq:linearizedtestp}, we obtain
  \begin{equation}\label{eq:Ju}
  \mu \int_{\Omega} \varepsilon S'(\overline{u},h) : \varepsilon p +\nu \int_{\Omega} \mathsf{m}'\left( \varepsilon \overline{y};\varepsilon S'(\overline{u},h) \right)   : \varepsilon p = -\alpha \int_{\Omega} h \cdot \overline{u} = - J_u (\overline{y}, \overline{u})h.
  \end{equation}
  Now, from the primal optimality condition \eqref{eq:Bstationarity}, we have \(J_y (\overline{y}, \overline{u})S'(\overline{u},h)\geq - J_u (\overline{y}, \overline{u})h  \), for every \(h \in U\), which, thanks to \eqref{eq:JyinderivativeofS} and \eqref{eq:Ju}, implies
  \begin{multline*}
  \mu \int_{\Omega} \varepsilon p : \varepsilon S'(\overline{u},h) +\nu \int_\Omega \lambda :  \varepsilon S'(\overline{u},h) \geq \mu \int_{\Omega} \varepsilon S'(\overline{u},h) : \varepsilon p \\+\nu \int_{\Omega} \mathsf{m}'\left( \varepsilon \overline{y};\varepsilon S'(\overline{u},h) \right)   : \varepsilon p \quad \forall h \in U.
  \end{multline*}
  This equation is equivalent to
  \begin{equation}\label{eq:ineqforlambdaandp}
  \nu \int_{\Omega} \lambda :  \varepsilon S'(\overline{u},h) \geq \nu \int_{\Omega} \mathsf{m}'\left( \varepsilon \overline{y};\varepsilon S'(\overline{u},h) \right)   : \varepsilon p, \quad \forall h \in U.
  \end{equation}

  Let now \(\xi \in Y\) be arbitrary but fixed. By Lemma \ref{th:densedirderivative}, there exists a sequence \((h_n )_{n\in \mathbb{N}} \) in \(U\) such that \(S'(\overline{u},h_n) \rightarrow \xi\) in \(Y\), as \(n \rightarrow \infty\). We have that \eqref{eq:ineqforlambdaandp} holds for \(h=h_n\), i.e.,
  \begin{equation}\label{eq:ineqforlambdap}
    \nu \int_{\Omega} \lambda : \varepsilon S'(\overline{u},h_n) \geq \nu \int_{\Omega} \mathsf{m}'\left( \varepsilon \overline{y};\varepsilon S'(\overline{u},h_n) \right)   : \varepsilon p \, dx,  \quad \forall n \in \mathbb{N}.
  \end{equation}
  Since the directional derivative of \(\mathsf{m}\) is Lipschitz continuous with respect to the direction and the operator \(\varepsilon\) is continuous, passing to the limit as \(n \rightarrow \infty\) in \eqref{eq:ineqforlambdap} yields
  \begin{equation} \label{eq:lambdaxip_ineq}
    \int_{\Omega} \lambda :  \varepsilon \xi \geq  \int_{\Omega} \mathsf{m}'\left( \varepsilon \overline{y};\varepsilon \xi \right)   : \varepsilon p \, dx ,\quad \forall \xi \in Y.
  \end{equation}
  Utilizing the directional derivative of \(\mathsf{m}\), \eqref{eq:lambdaxip_ineq} can be written as
  \begin{multline}\label{eq:ineqthreesets}
    \int_{\Omega} \lambda :  \varepsilon \xi \geq \int_{\{\abs{\varepsilon \overline{y}} > g\}} \left[ \varepsilon \xi+g \left( \dfrac{\varepsilon \overline{y}}{\abs{\varepsilon \overline{y}}^3} :\varepsilon \xi \right) \varepsilon \overline{y} -g\dfrac{\varepsilon \xi }{\abs{\varepsilon \overline{y} }} \right]: \varepsilon p\\
  + \int_{\{\abs{\varepsilon \overline{y}} = g\}} \frac{1}{g^2} \max(0,(\varepsilon \overline{y}: \varepsilon \xi)) \varepsilon \overline{y}:\varepsilon p, \quad \forall \xi \in Y.
\end{multline}
  Thanks to \eqref{eq:stronglinearizedlambda1} and \eqref{eq:stronglinearizedlambda2}, we thus obtain that
  \begin{equation*}
    \int_{\mathcal A} \lambda :  \varepsilon \xi \geq \int_{\mathcal A} \frac{1}{g^2} \max(0,(\varepsilon \overline{y}: \varepsilon \xi)) \varepsilon \overline{y}:\varepsilon p, \quad \forall \xi \in Y.
  \end{equation*}
\end{proof}

The result of Theorem \ref{th:strongstationarity} is stronger than the weak stationarity result in the sense that condition \eqref{eq:stronglinearizedlambda3}, which can be considered as a condition on the sign of \(\lambda\), is not present in the optimality system \eqref{eq: weak system}.
Moreover, Theorem \ref{th:strongstationarity} is equivalent to the Bouligand stationarity, which will be showed in the following proposition.

\begin{theorem}
  Let \((\overline{y}, \overline{u})\in Y \times U\), \(p\in Y\) and \(\lambda \in \Lmat \) satisfy \eqref{eq:strongOS}. Then, it also satisfies the optimality condition \eqref{eq:Bstationarity}.
\end{theorem}
\begin{proof}
  Let \(h \in U\) be arbitrary. Let us set \(z=S'(\overline{u};h)\). Testing  \eqref{eq:stronglinearizedu}, with \(h\), we obtain
  \begin{equation*}
    -\alpha \left( \overline{u},h \right)=(p,h), \quad \forall h \in U.
  \end{equation*}
  Using this on the right hand side of equation \eqref{eq:linearizedeq}, we obtain
  \begin{equation*}
    -\alpha \left( \overline{u},h \right) = \mu \int_{\Omega} \varepsilon z : \varepsilon p \, dx +\nu \int_{\Omega} \mathsf{m}'\left( \varepsilon \overline{y};\varepsilon z \right)   : \varepsilon p \, dx.
  \end{equation*}
  Combining with \eqref{eq:stronglinearizedy} we then obtain
  \begin{equation*}
    -\alpha \int_\Omega \overline{u}\cdot h \, dx =- \int_{\Omega} \lambda :  \varepsilon z + \int_\Omega (\overline{y}-z_d)\cdot z \, dx +\nu\int_{\Omega}   \mathsf{m}'\left( \varepsilon \overline{y};\varepsilon z \right): \varepsilon p \, dx.
  \end{equation*}
  Utilizing \eqref{eq:stronglinearizedlambda1}-\eqref{eq:stronglinearizedlambda3} it then follows that
  \begin{equation*}
    \int_{\Omega} \lambda :  \varepsilon \xi \geq  \int_{\Omega} \mathsf{m}'\left( \varepsilon \overline{y};\varepsilon \xi \right)   : \varepsilon p \, dx ,\quad \forall \xi \in Y.
  \end{equation*}
  and, consequently,
  \begin{equation*}
    -\alpha \int_\Omega \overline{u}\cdot h \, dx \leq \int_\Omega (\overline{y}-z_d)\cdot z \, dx,
  \end{equation*}
  which corresponds to the optimality condition \eqref{eq:Bstationarity}.
\end{proof}

\end{document}